\documentclass[letterpaper, DIV=10]{scrartcl}

\usepackage{amssymb, amsthm, mathtools, cite}
\usepackage[hidelinks, citecolor=blue]{hyperref}
\hypersetup{colorlinks=true, urlcolor = blue, linkcolor=black}

\newtheorem{theorem}{Theorem}
\newtheorem{corollary}[theorem]{Corollary}
\newtheorem{lemma}[theorem]{Lemma}
\newtheorem{proposition}[theorem]{Proposition}
\numberwithin{equation}{section}
\numberwithin{theorem}{section}

\newcommand{\rr}{{\mathbb{R}}}

\newcommand{\zz}{{\mathbb{Z}}}
\newcommand{\nn}{{\mathbb{N}_0}}
\newcommand{\cp}{{\mathbb{C}_+}}
\newcommand{\ee}{{\mathbb{E}\,}}
\newcommand{\eesub}[1]{{\mathbb{E}_{#1}\,}}
\newcommand{\pp}{{\mathbb{P}}}
\newcommand{\oh}{{\mathcal{O}}}
\newcommand{\im}{{\operatorname{Im}\,}}
\newcommand{\re}{{\operatorname{Re }\,}}
\newcommand{\parder}[2]{{\frac{\partial #1}{\partial #2}}}
\newcommand{\tr}{{\operatorname{Tr}\,}}
\newcommand{\sgn}{{\operatorname{sgn}\,}}

\newcommand{\ketbra}[2]{{{|#1\rangle\langle #2|}}}
\newcommand{\beq}[1]{\begin{equation} \label{#1}}
\newcommand{\eeq}{\end{equation}}
\newcommand{\zel}{{z_\ell}}

\begin{document}
\addtokomafont{author}{\raggedright}
\title{\raggedright The Phase Transition in the Ultrametric Ensemble and Local Stability of Dyson Brownian Motion}
\author{\hspace{-.075in}Per von Soosten and Simone Warzel}
\date{\vspace{-.2in}}
\maketitle
\minisec{Abstract}
We study the ultrametric random matrix ensemble, whose independent entries have variances decaying exponentially in the metric induced by the tree topology on $\mathbb{N}$, and map out the entire localization regime in terms of eigenfunction localization and Poisson statistics. Our results complement existing works on complete delocalization and random matrix universality, thereby proving the existence of a phase transition in this model. In the simpler case of the Rosenzweig-Porter model, the analysis yields a complete characterization of the transition in the local statistics. The proofs are based on the flow of the resolvents of matrices with a random diagonal component under Dyson Brownian motion, for which we establish submicroscopic stability results for short times. These results go beyond norm-based continuity arguments for Dyson Brownian motion and complement the existing analysis after the local equilibration time.
\bigskip

\section{Introduction}
One-dimensional lattice Hamiltonians with random long-range hopping provide a useful and simplified testing ground for the Anderson metal-insulator transition in more complicated systems. Two prominent examples of such models are the random band matrices \cite{PhysRevLett.64.1851,PhysRevLett.67.2405}, whose entries $H_{ij}$ are zero outside some band $|i-j| \le W$, and the power-law random band matrices (PRBM) \cite{RevModPhys.80.1355,PhysRevE.54.3221}, whose entries $H_{ij}$ have variances decaying according to some power of the Euclidean distance $|i-j|$. Even for these models, the mathematically rigorous understanding is far from complete, although there has been some progress; see \cite{MR2726110,0036-0279-66-3-R02, MR3695802,MR1942858,MR3085669, MR2525652} and references therein. This article is concerned with a further simplification, the ultrametric ensemble of Fyodorov, Ossipov and Rodriguez \cite{1742-5468-2009-12-L12001}, which is essentially obtained by replacing the Euclidean distance in the definition of the PRBM with the metric induced by the tree topology.

The index space of the ultrametric ensemble is $B_n = \{1,2,\dots, 2^n\}$ endowed with the metric
\[d(x,y) = \min  \left\{ r \geq 0 \, | \, \mbox{$x$ and $y$ lie in a common member of $\mathcal{P}_r$} \right\},\]
where $\{\mathcal{P}_r\}$ is the nested sequence of partitions defined by
\[B_n = \{1, \dots, 2^{r}\} \cup \{2^r + 1, \dots, 2\cdot 2^r \} \cup \dots \cup \{2^{n-r-1}2^r + 1, \dots, 2^n\}.\]
The basic building blocks of the ultrametric ensemble are the matrices $\Phi_{n,r}: \ell^2(B_n) \to \ell^2(B_n)$ whose entries are independent (up to the symmetry constraint) centered real Gaussian random variables with variance
\beq{eq:Phidef} \ee \left| \langle \delta_y, \Phi_{n,r} \delta_x \rangle \right|^2 = 2^{-r} \begin{cases} 2 & \mbox{ if } d(x,y) = 0\\ 1 & \mbox{ if } 1 \le d(x,y) \le r\\ 0 & \mbox{ otherwise. }\end{cases}
\eeq
Here and throughout this paper, $\delta_x \in \ell^2$ denotes the canonical basis element defined by
\[\delta_x(u) = \begin{cases} 1 & \mbox{ if }  u = x\\ 0 & \mbox{ if } u \neq x\end{cases}\]
and $\delta_{xy} = \langle \delta_y, \delta_x \rangle$. Thus $\Phi_{n,r}$ is a direct sum of $2^{n-r}$ random matrices drawn independently from the Gaussian Orthogonal Ensemble (GOE) of size $2^r$. The ultrametric ensemble with parameter $c \in \rr$ refers to the random matrix
\beq{eq:Hdef}H_n = \frac{1}{Z_{n,c}} \sum_{r =0}^n 2^{-\frac{(1+c)}{2}r} \Phi_{n,r}
\eeq
where $\Phi_{n,r}$ and $\Phi_{n,s}$ are independent for $r \neq s$. We choose the normalizing constant $Z_{n,c}$ such that
\[\sum_{y \in B_n} \ee \left| \langle \delta_y, H_n \delta_x \rangle \right|^2 = 1,\]
making the variance matrix $\Sigma_n$ of $H_n$ doubly stochastic. The original definition in \cite{1742-5468-2009-12-L12001} contains an additional parameter governing the relative strengths of the diagonal and off-diagonal disorder, but this parameter does not significantly alter our analysis and so we omit it altogether. Moreover, the authors of \cite{1742-5468-2009-12-L12001} constructed the block matrices $\Phi_{n,r}$ from the Gaussian Unitary Ensemble (GUE) and our results apply to both GOE and GUE blocks with only slight changes.

The ultrametric ensemble is a hierarchical analogue of the PRBM in a sense which was first introduced for the Ising model by Dyson~\cite{MR0436850} and studied rigorously in the context of quantum hopping systems with only diagonal disorder in~\cite{MR1063180, MR2352276, MR2413200, MR2909106,MR1463464, MR3649447}. In particular, the definition~\eqref{eq:Hdef} shows that the variance matrix $\Sigma_n$ is a rescaled and shifted version of the hierarchical Laplacian, for which one can easily show that $\Sigma_n \geq 0$ by explicit diagonalization as in~\cite{MR2276652}. The normalizing factor $Z_{n,c}$ can also be calculated
\begin{align*}Z_{n,c}^2 &=  \sum_{y \in B_n} \ee \left| \langle \delta_y, \left(\sum_{r =0}^n 2^{-\frac{(1+c)}{2}r} \Phi_{n,r}\right) \delta_x\rangle \right|^2 =  \left(1 - 2^{-(1+c)(n+1)} \right)\,  \frac{1 + \oh(1)}{1-2^{-(1+c)}}\\
&= \begin{cases} \oh(1) & \mbox{if } c > -1 \\ \oh\left(2^{-(1+c)n}\right) & \mbox{if } c < -1 \end{cases}
\end{align*}
so that $Z_{n,c}$ grows exponentially in $n$ in case $c < -1$ and $Z_{n,c}$ is asymptotically constant in case $ c > - 1 $. Finally, the spread
\beq{eq:spread}M_n \vcentcolon= \left(\max_{x,y \in B_n} \, \ee |\langle \delta_y, H_n\delta_x \rangle|^2 \right)^{-1} =  \begin{cases}Z_{n,c}^2 \, 2^{- o(n)}    & \mbox{if } c \geq -2\\   2^{ (1+o(1))n} & \mbox{if} \;  c < -2 \end{cases}\eeq
also grows like a positive power of the system size $2^n$ when $c < -1$.

Because of the successful track record of hierarchical approximations in statistical physics, one might expect the core features of the PRBM phase transition to be present in the ultrametric ensemble as well. Indeed, the authors of~\cite{1742-5468-2009-12-L12001} present arguments at a theoretical physics level of rigor as well as numerical evidence for a localization-delocalization transition in the eigenfunctions of $H_n$ as the parameter changes from $c > 0$ to $c < 0$. In this paper, we pursue the point of view that the effect of the Gaussian perturbations $\Phi_{n,r}$ on the spectrum of $H_n$ can be described dynamically by Dyson Brownian motion~\cite{MR0148397} and, in this sense, the critical point $c=0$ is natural because it governs whether the evolution passes the local equilibration time of this system or not. In particular, we establish the localized phase by proving that the eigenfunctions remain localized and the level statistics converge to a Poisson point process if $c > 0$. In comparison to both the random band matrices and the PRBM, the hierarchical structure therefore significantly facilitates the analysis. We are not aware of any rigorous results for the PRBM, whereas for random band matrices the only localization result is due to Schenker~\cite{MR2525652} and no proof of Poisson statistics is known.

For the first result we recall the Wegner estimate~\cite{MR639135}, which asserts that the infinite-volume density of states measure defined by
\beq{eq:dosdef}\nu(f) = \lim_{n \to \infty} 2^{-n} \sum_{\lambda_j \in \sigma(H_n)} f(\lambda_j)
\eeq
has a bounded Radon-Nikodym derivative whose values we denote by $\nu(E)$.
\begin{theorem}[Poisson statistics] \label{thm:poisson} Suppose $c > 0$ and let $E \in \rr$ be a Lebesgue point of $\nu$. Then, the random measure
\[\mu_n(f) = \sum_{\lambda \in \sigma(H_n)} f(2^n(\lambda - E))\]
converges in distribution to a Poisson point process with intensity $\nu(E)$ as $n \to \infty$.
\end{theorem}
The proof of Theorem \ref{thm:poisson} is contained in Section \ref{sec:poisson}.

The second main result says that if an eigenfunction of $H_n$ in some mesoscopic energy interval has any mass at some $x \in B_n$, then actually all but an exponentially small amount of the total mass is carried in an exponentially vanishing fraction of the volume near $x$ with high probability. We make this precise in terms of the eigenfunction correlator
\[Q_n(x,y;W) = \sum_{\lambda \in \sigma(H_n) \cap W} |\psi_\lambda(x)\psi_\lambda(y)|,\]
which in completely delocalized regimes is typically given by
$  \sum_{y \neq x} Q_n(x,y;W) \approx 2^n|W| $.
Thus, since $2^n|W|$ grows very large for mesoscopic spectral windows, it is a signature of localization if the eigenfunction correlator asymptotically vanishes for small enough mesoscopic intervals $W$, as is proved in the following theorem.
\begin{theorem}[Eigenfunction localization]\label{thm:localization} Suppose $c > 0$ and let $E \in \rr$. Then, there exist $w, \mu, \kappa > 0$, $C<\infty$, and a sequence $m_n$ with $n - m_n \to \infty$ such that for every $x \in B_n$ the $ \ell^2$-normalized eigenfunctions satisfy
\[\pp\left( \sum_{y  \in B_n \setminus B_{m_n}(x)} Q_n(x,y;W) > 2^{-\mu n} \right) \le C \,2^{-\kappa n}\]
with
\[W = \left[E - 2^{-(1-w)n}, E + 2^{-(1-w)n}\right].\]
\end{theorem}
The proof of Theorem \ref{thm:localization} in Section \ref{sec:localization} makes use of a new relation (Theorem \ref{thm:qbound}) between the eigenfunction correlator and the imaginary part of the Green function at complex energies.

The above results gain additional interest upon noting that, when $c < -1$, the ultrametric ensemble has an essential mean field character and techniques originally developed for Wigner matrices show that the energy levels agree asymptotically with those of the GOE and that the eigenfunctions are completely delocalized. We will now roughly sketch how to apply these results in the present situation and state the corresponding theorems. The results of \cite{MR3068390} show that the semicircle law $\rho_{sc}(E) = \sqrt{(4-E^2)_+}/(2\pi)$ is valid on scales of order $M_n^{-1}$ (c.f.~\eqref{eq:spread}) also for the matrices
\[\widetilde{H}_n =  \frac{1}{Z_{n,c}} \sum_{r =0}^{n-1} 2^{-\frac{(1+c)}{2}r} \Phi_{n,r} + \frac{1-\sqrt{T_n}}{Z_{n,c}} 2^{-\frac{(1+c)}{2}n} \Phi_{n,n} \]
with a small part of the final $\oh(1)$ Gaussian component removed. We set $ T_n = M_n^{-1+\delta} $ with $ \delta \in (0,1) $. The validity of the local semicircle law already implies the complete delocalization of the eigenfunctions in mesoscopic windows in the bulk of the spectrum using the results of~\cite[Thm. 2.21]{0036-0279-66-3-R02}. The assumption of~\cite{0036-0279-66-3-R02} that the spectrum of the variance matrix is well-separated from $-1$ follows simply from $\Sigma_n \geq 0$.

\begin{theorem}[Eigenfunction delocalization; cf. \cite{0036-0279-66-3-R02, MR3068390}] Let $c < -1$. For any compact interval $ I \subset (-2,2) $ there exist $\kappa, \epsilon > 0$ such that for all $ E \in I $ the $ \ell^2$-normalized eigenfunctions of $H_n$ in $[E-M_n^{-1}, E+M_n^{-1}]$ satisfy
\[\|\psi_{\lambda}\|_\infty =\oh(M_n^{-1/2 + \epsilon})\]
with probability $1 - \oh(N^{-\kappa})$.
\end{theorem}

Random matrix universality of the local statistics may be expressed by saying that the $k$-point correlation functions
\[\rho^{(k)}_{H_n}(\lambda_1, .. \lambda_k) =  \int_{\rr^{2^n - k}} \! \rho_{H_n} (\lambda_1, \dots, \lambda_{2^n}) \, d\lambda_{k+1} \dots \, d\lambda_{2^n},\]
the $k$-th marginals of the symmetrized eigenvalue density $\rho_{H_n}$, locally agree with the corresponding objects for the GOE asymptotically. For this, we employ the work of Landon, Sosoe and Yau \cite[Thm.~2.2]{landonsosoeyauarxiv} concerning the universality of Gaussian perturbations for
\[H_n = \widetilde{H}_n +  \frac{\sqrt{T_n}}{Z_{n,c}} 2^{-\frac{(1+c)}{2}n} \Phi_{n,n}.\]
For the statement of the theorem, let
\begin{align*}\Psi_{n,E}^{(k)}(\alpha_1,\dots,\alpha_k) &= \rho^{(k)}_{H_n}\left( E + 2^{-n} \frac{\alpha_1}{\rho_{sc}(E)}, \dots, E + 2^{-n} \frac{\alpha_k}{\rho_{sc}(E)} \right)\\
 &-  \rho^{(k)}_{GOE}\left( E + 2^{-n} \frac{\alpha_1}{\rho_{sc}(E)}, \dots, E + 2^{-n} \frac{\alpha_k}{\rho_{sc}(E)} \right),
\end{align*}
where $\rho^{(k)}_{GOE}$ is the $k$-point correlation function of the $2^n \times 2^n$ GOE and $\rho_{sc}$ is the density of the semicircle law.
\begin{theorem}[WDM statistics; cf.~\cite{landonsosoeyauarxiv,MR3068390}] Suppose $c < -1$, $E \in (-2,2)$ and $k \geq 1$. Then,
\[\lim_{n \to \infty} \int_{\rr^k} \! O(\alpha) \Psi_{n,E}^{(k)}(\alpha) \, d\alpha = 0\]
for every $O \in C_c^\infty(\rr^k)$.
\end{theorem}

Summing up, these results rigorously prove the existence of a metal-insulator transition in the ensemble of ultrametric random matrices. In particular, our results allow an approach all the way to the critical point from the localized side $ c > 0 $. The previous arguments do not cover the regime $ c \in [-1,0) $. In this case, the eigenfunctions are still conjectured to be completely delocalized and the local eigenvalue statistics are conjectured to remain in the Wigner-Dyson-Mehta universality class as in the case $ c < -1 $ \cite{1742-5468-2009-12-L12001}.

The Rosenzweig-Porter model is the $N \times N$ random matrix
\[H_t = V + \sqrt{t} \, \Phi_{n,n} \]
with $ N = 2^n  $, interpolation parameter $ t = N^{-(1+c)}$, and independent random potential
\beq{pot} V = \sum_{x} V(x) \ketbra{\delta_x}{\delta_x}
\eeq
whose entries are drawn independently from some density $\varrho \in L^\infty$. It provides a standard interpolation between integrability and chaos and more recently has also been suggested as a toy model for many-body localization with three distinct phases~\cite{0295-5075-115-4-47003, 1367-2630-17-12-122002}. As a by-product of our analysis, we completely characterize the localized phase $c> 0$ of this model, which may be thought of as a hierarchical model with the ``intermediate layers'' removed.

\begin{theorem}[Rosenzweig-Porter model]\label{thm:rpthm} Suppose $t \le N^{-(1+c)}$ with $c > 0$ and let $E \in \rr$. Then:
\begin{enumerate}
\item As $N \to \infty$, the random measure defined by
\[\mu_N(f) = \sum_{\lambda \in \sigma(H_t)} f(N(\lambda - E))\]
converges in distribution to a Poisson point process with intensity $\varrho(E)$ provided $E$ is a Lebesgue point of $\varrho$.
\item There exist $w, \mu, \kappa > 0$ and $C<\infty$ such that for every $x \in \{1,\dots,N\}$ the $\ell^2$-normalized eigenfunctions satisfy
\[\pp\left(\sum_{\lambda \in \sigma(H_t) \cap W}  \sum_{y \neq x}|\psi_\lambda(x)\psi_\lambda(y)| > N^{-\mu} \right) \le C N^{-\kappa}\]
with
\[W = \left[E - N^{-(1-w)}, E + N^{-(1-w)}\right].\]
\end{enumerate}
\end{theorem}

A proof of this result, which is similar to but simpler than those of Theorems~\ref{thm:poisson} and~\ref{thm:localization}, is included in Appendix~\ref{sec:RosPorter}.

If $ t = N^{-(1+c)} $ with  $c < 0$ the prerequisites of~\cite[Thm.~2.2]{landonsosoeyauarxiv} may be verified for the Rosenzweig-Porter model by an exponential moment calculation similar to Cram\'{e}r's theorem. This proves the emergence of Wigner-Dyson-Mehta statistics in the fixed-energy sense. The first point of Theorem~\ref{thm:rpthm} thus optimally complements these results and completes the mathematical understanding of the phase transition in the Rosenzweig-Porter model in terms of the local statistical behavior of the energy levels. The second point proves localization in the same sense as Theorem~\ref{thm:localization}, and also yields an explicit relation between $w,\mu,\kappa$ and $c$, which shows that $w,\mu,\kappa$ may increase if $c$ increases as well. If $c \le -1$, the complete delocalization of the eigenfunctions was proved by Lee and Schnelli~\cite{MR3134604} as a corollary to a local law. In terms of the eigenvalue statistics, the non-ergodic delocalized phase established in~\cite{nonergodic,benignibourgade} can presumably only be detected by mesoscopic Poissonian fluctuations around the microscopic Wigner-Dyson-Mehta statistics (see also \cite{duitsjohansson,landonhuangarxiv}).

\section{Local Stability of Dyson Brownian Motion}
In the technical core of this paper, we consider a general $N \times N$ random matrix flow
\beq{eq:flowdef} H_t = T + V + \Phi_t,
\eeq
where $\Phi_t$ is a Brownian motion in the space of symmetric matrices whose entries are given by
\beq{eq:entrydef} \langle \delta_y, \Phi_t \delta_x \rangle = \sqrt{\frac{1 + \delta_{xy}}{N}}B_{xy}(t)
\eeq
with independent standard Brownian motions $B_{xy}(t) = B_{yx}(t)$. The random potential $ V $ is given by \eqref{pot} in terms of random variables $ \{ V(x) \} $ which are independent of the Brownian motions and whose conditional distributions are uniformly Lip\-schitz continuous, i.e.,
\beq{eq:potassumption}\pp\left(V(x) \in I \, | \, \{V(y)\}_{y \neq x} \right) \le C_V|I|,
\eeq
for all Borel sets $I \subset \rr$ and $x \in \{1,\dots,N\}$ with a constant $C_V < \infty$ independent of $N$. Finally, $T$ is some real symmetric $N \times N$ matrix, which may also be random provided $T$, $V$, and $\Phi_t$ remain independent. Dyson derived the equations
\[d\lambda_j(t) = \sqrt{\frac{2}{N}}dB_j(t) + \frac{1}{N}\sum_{i \neq j} \frac{dt}{\lambda_j(t) - \lambda_i(t)}\]
for the evolution of the eigenvalues $\lambda_1(t) \le \dots \le \lambda_N(t)$ and conjectured $t = N^{-1}$ as the local equilibration time of this system. Thus, one expects that the local statistics of the eigenvalues agree asymptotically as $N \to \infty$ with those of the GOE if $t \gg N^{-1}$. This was first proved by Erd\H{o}s, Schlein, and Yau~\cite{MR2810797} when $H_0$ is an independent Wigner matrix and recently for some very general, even deterministic, initial conditions by Landon, Sosoe, and Yau~\cite{landonsosoeyauarxiv, landonyauarxiv}. In particular, the result  \cite{landonsosoeyauarxiv}, which establishes fixed-energy universality  in the bulk of the spectrum for $ t  \gg N^{-1}$, covers the case $ T = 0 $ with asymptotically full probability. The eigenfunctions of $H_t$ also follow a highly singular stochastic differential equation, which has been studied for Wigner initial conditions by Bourgade and Yau~\cite{MR3606475}.

All of the aforementioned works rely on powerful rigidity estimates for the eigenvalues in the regime $t \gg N^{-1}$, which are not available for $t \ll N^{-1}$, where the spectral characteristics of $H_0 = T+V$ are expected to remain dominant. In this article, we study the stability of the spectral measures in the regime $t \ll N^{-1}$ by deriving a more tractable stochastic differential equation for the resolvent
\[R_t(z) = (H_t - z)^{-1},\]
which is significantly more amenable to analysis but still carries the relevant spectral information in its entries. In place of the rigidity estimates, the enabling tool for the analysis of the resolvent flow is the smoothing of spectral quantities with the external potential via spectral averaging (see Section~\ref{sec:flowsec} for details).

Our first result for $H_t$ is about the normalized trace
\[S_t(z) = \frac{1}{N} \tr R_t(z) = \int \! \frac{1}{\lambda - z} \, \nu_t(d \lambda),\]
which we have written as the Stieltjes transform of the empirical eigenvalue measure
\[\nu_t = \frac{1}{N} \sum_{\lambda \in \sigma(H_t)} \delta_{\lambda}.\]
Thus $S_t(z)$ contains detailed local information about the eigenvalues of $H_t$ as $z$ approaches the real axis. In particular, since the mean eigenvalue spacing of $H_t$ is typically of order $N^{-1}$, knowledge of $S_t(z)$ with $\im z \approx N^{-1}$ makes it possible to track individual eigenvalues near $\re z$ along the flow~\eqref{eq:flowdef}. The following theorem shows that $S_t(z)$ remains stable even when $\im z \ll N^{-1}$ provided that also $t \ll N^{-1}$.

\begin{theorem} \label{thm:stabthmtr} For every $c > 0$, there exists $C < \infty$, depending only on $c$ and $C_V$, such that
\[\ee \left|S_t\left(E + i\eta \right) - S_0\left(E + i\eta\right)\right| \le C  N^{-c/2} \left(1 + \frac{1}{N\eta} + \frac{1}{(N\eta)^3}\right)\]
for all $t \le N^{-(1 + c)}$ and $E \in \rr$.
\end{theorem}

In essence, Theorem~\ref{thm:stabthmtr} asserts that for any spectral scale much larger than $t \ll N^{-1}$ the empirical eigenvalue measure is unaffected by the flow~\eqref{eq:flowdef}. This bound is much stronger than the one obtained from the crude norm estimate
\beq{eq:crudenorm} \|R_t(z) - R_0(z)\| \le \|R_t(z) \Phi_t R_0(z)\| \approx \frac{\sqrt{t}}{(\im z)^2}
\eeq
and potential improvements on this theme obtained by running the Brownian motions $\langle \delta_y, \Phi_t \delta_x \rangle$ one after the other. The example $H_0 = 0$ shows that at least some regularity of the initial condition is needed for Theorem~\ref{thm:stabthmtr} to remain true. However, Theorem~\ref{thm:stabthmtr} can be proved for slightly more general $H_0$ under the weaker assumption that $H_t$ satisfies the Wegner and Minami~\cite{MR1385082} estimates
\[\ee \nu_t(I) \le C|I|, \qquad \ee \nu_t(I) (\nu_t(I) - N^{-1}) \le C|I|^2\]
with a constant $C < \infty$ uniform in $N$ and $t$. This is easily seen from the proof below. It is also possible to present Theorem~\ref{thm:stabthmtr} (and Theorem~\ref{thm:stabthm}) as explicit bounds for arbitrary $t > 0$, but we artificially restrict to $t \le N^{-(1+c)}$ in order to keep the right hand side simple.

The properties of the eigenfunctions of $H_t$ are encoded in the spectral measures
\[\mu_{xy} = \sum_{\lambda \in \sigma(H_t)} \psi_\lambda(x) \overline{\psi}_\lambda(y) \delta_\lambda\]
where $\{\psi_\lambda\}$ is an orthonormal basis of eigenfunctions of $H_t$ and we have eased the notational burden by keeping the dependence of $\psi_\lambda$ and $\mu_{xy}$ on $t$ implicit. Hence, the Green functions
\[G_t(x,y; z) =  \langle \delta_y, R_t(z) \delta_ x \rangle = \int \! \frac{1}{\lambda - z} \, \mu_{xy}(d\lambda)\]
at scales $\im z \approx N^{-1}$ describe the eigenfunctions of $H_t$ locally near $\re z$. The stability result analogous to Theorem~\ref{thm:stabthmtr} for $G_t(x,y;z)$ is contained in the following theorem.

\begin{theorem} \label{thm:stabthm} For every $c > 0$, there exists $C < \infty$, depending only on $c$ and $C_V$, such that
\[ \frac{1}{N}\sum_y \ee \left| G_t\left(x,y; E + i\eta\right)-  G_0\left(x,y;E + i\eta \right) \right| \le C  N^{-c/2} \left(1 + \frac{1}{N\eta} + \frac{1}{(N\eta)^3}\right)\]
for all $t \le N^{-(1 + c)}$, $E \in \rr$, and  $x \in \{1,\dots,N\}$.
\end{theorem}

The proofs of Theorems~\ref{thm:stabthmtr} and~\ref{thm:stabthm}, which may be found in Section~\ref{sec:flowsec}, are based on the fact that $R_t(z)$ satisfies the stochastic differential equation
\beq{eq:approxflow}dR_t(z) = \left(S_t(z) \parder{}{z} R_t(z) + \frac{1}{2N} \parder{^2}{z^2} R_t(z) \right) \, dt + d\tilde{M_t},
\eeq
where $\tilde{M}_t$ is a matrix-valued martingale whose entries can be given explicitly in terms of $R_t(z)$. The details of this equation, as well as its derivation, are contained in Section~\ref{sec:flowsec}. The proofs rely on bounds controlling the regularized singularities of the terms in the right-hand side of~\eqref{eq:approxflow} down to the scale $ \eta = \mathcal{O}(N^{-1}) $. In the same vein as remarked below~\eqref{eq:crudenorm}, we stress that Theorems~\ref{thm:stabthmtr} and~\ref{thm:stabthm} cannot be obtained from the trivial estimate $ \|  \frac{\partial^n}{\partial z^n} R_t(z) \| \leq (\im z)^{-(n+1)} $. Instead, we use the smoothing of singularities through the external random potential $ V $ as detailed in Section~\ref{sec:smooth}. 
Since we are interested in times scales $ t \ll N^{-1} $, the potential is in fact the only possible source for this smoothing.

Let us conclude by noting that the analogue of~\eqref{eq:flowdef} for perturbations drawn from the GUE,
\[\langle \delta_y, \tilde{\Phi}_t \delta_x \rangle = \sqrt{\frac{1}{N}}\begin{cases} \frac{1}{\sqrt{2}} (B_{xy}(t) + i\tilde{B}_{xy}(t)) & \mbox{ if } x < y\\ B_{xx}(t) & \mbox{ if } x=y \end{cases}\]
with $\tilde{B}_{xy}$ independent of $B_{xy}$, has also been widely studied. The analysis of this model is usually simpler because the additional symmetry enables explicit integration formulas (see~\cite{0036-0279-66-3-R02} and references therein for a summary) and all the results and methods of this paper require only minor modifications to treat also the GUE flow.

\section{Smoothing Effects of the Potential} \label{sec:smooth}
Throughout this section we will let $H$ be a general $N \times N$ random matrix of the form
\[H = T + V,\]
where $V$ is a potential satisfying the assumption~\eqref{eq:potassumption} and $T$ is some Hermitian random matrix independent of $V$, which should be thought of as $T + \Phi_t$ from~\eqref{eq:flowdef}. Our goal is to use the smoothing effects of $V$ on the spectral measures $\mu_{xy}$ of $\delta_x$ and $\delta_y$ for $H$ and the empirical eigenvalue measure
\[\nu(f) = \frac{1}{N}\sum_{\lambda \in \sigma(H)} f(\lambda)\]
to control the resolvent flow~\eqref{eq:approxflow}. We start by recalling two staples of the theory of random Schr\"{o}dinger operators, the spectral averaging principle~\cite{MR841099} and the closely related Wegner estimate, whose proofs may be found in~\cite{MR3364516}. The former asserts that for any Borel set $I \subset \rr$ and any $x \in \{1,\dots,N\}$ we have
\beq{eq:specavg}\eesub{x}\left[ \mu_{x}(I) \right] \le C_V|I|,
\eeq
where $\eesub{x}$ denotes the conditional expectation with respect to the random variables $\{V(k): k \neq x\}$. By averaging this bound over all $x \in \{1,\dots,N\}$, we immediately obtain the latter result, namely that
\beq{eq:wegnerest} \ee \nu(I)  \le C_V |I|
\eeq
for all Borel sets $I \subset \rr$. The following lemma is a simple extension of these results based on the proof of Minami's estimate by Combes, Germinet, and Klein~\cite{MR2505733}. We write $|\mu|$ for the total variation measure of $\mu$.
\begin{lemma} \label{thm:specavgext} There exists $C < \infty$, depending only on $C_V$, such that
\begin{enumerate}
\item $\ee |\mu_{xy}|(I) \le C|I|$ and
\item $\ee \left[ \nu(I) |\mu_{xy}|(J)\right]  \le C \left(|I| + \frac{2}{N}\right) |J|$
\end{enumerate}
for all Borel sets $I,J \subset \rr$ and $x,y \in \{1,\dots,N\}$.
\end{lemma}
\begin{proof}
Notice that
\[|\mu_{xy}|(I) = \sum_{\lambda \in \sigma(H) \cap I} \left |\psi_\lambda(x) \psi_\lambda(y)\right|,\]
so the Cauchy-Schwarz inequality implies
\[|\mu_{xy}|(I)\le \sqrt{ \mu_x(I) \mu_y(I)}.\]
Applying the Cauchy-Schwarz inequality to the expectation $\eesub{xy}$ conditioned on $\{V(k): k\neq x,y\}$ and using~\eqref{eq:specavg} then yield
\beq{eq:offdiagspecavg}\eesub{xy} |\mu_{xy}|(I) \le \eesub{xy}  \sqrt{ \mu_x(I) \mu_y(I)} \le \sqrt{\eesub{xy} \mu_x(I) \eesub{xy} \mu_y(I)} \le C|I|,
\eeq
which implies the first assertion of the Lemma.

For the second claim, notice that for fixed values $\{V(k): k\neq x,y\}$ of the potential away from $x$ and $y$, the number of eigenvalues in $I$ can change by at most two as $V(x)$ and $V(y)$ vary in $\rr$. Hence
\begin{align*}\ee \left[ \nu(I) |\mu_{xy}|(J)\right] &\le \ee \left[ \left( \nu(I) + \frac{2}{N} \right) \eesub{xy} |\mu_{xy}|(J)\right] \\
&\le C|J|  \ee \left[ \nu(I) + \frac{2}{N} \right] \le C \left(|I| + \frac{2}{N}\right) |J|,
\end{align*}
by~\eqref{eq:wegnerest} and~\eqref{eq:offdiagspecavg}.
\end{proof}

Intuitively, Lemma~\ref{thm:specavgext} asserts that the joint measure $\ee\left[\nu \times |\mu_{xy}|\right]$ is continuous down to scales of order $N^{-1}$, which clearly has consequences for the integrals of test functions in terms of their variations on scales of order $N^{-1}$. The next results are a quantitative manifestation of this idea for the Stieltjes transforms
\[G(x,y;z) = \int \! \frac{1}{\lambda - z} \, \mu_{xy}(d\lambda)\]
and
\[S(z) = \int \! \frac{1}{\lambda - z} \, \nu(d\lambda),\]
which occur naturally in our study of the resolvent flow. In particular, the following theorem gives bounds for the drift.

\begin{theorem}\label{thm:driftbounds} There exists $C < \infty$, depending only on $C_V$, such that
\[\ee \left|\frac{1}{2N} \parder{^2}{z^2} G(x,y; z) \right| \le \frac{C}{N (\im z)^2} \]
and
\[\ee \left| S(z) \parder{}{z} G(x,y; z) \right| \le C N \left( \log N + \frac{1}{N \, \im z} \right) \left(1 + \frac{1}{(N \, \im z)^{2}} \right) +  \frac{C}{\im z}\]
for all $x,y \in \{1,\dots,N\}$ and $z \in \cp$.
\end{theorem}
\begin{proof} The first point of Lemma~\ref{thm:specavgext} implies that
\begin{align} \label{eq:driftbound1} \ee \left| \frac{1}{2N} \parder{^2}{z^2} G(x,y; z) \right| &\le \frac{1}{N \, \im z} \ee \int \! \frac{1}{|\lambda-z|^2} \, |\mu_{x,y}|(d\lambda)\nonumber\\
&\le  \frac{C}{N\, \im z}\int \! \frac{1}{|\lambda-z|^2} \, d\lambda\nonumber\\
&\le \frac{C}{N (\im z)^2},
\end{align}
which is the first assertion of the theorem.

Next, let us introduce
\[f(\lambda) =\frac{1_{|\lambda - \re z| \le 1}}{|\lambda - z|}, \quad \tilde{f}(\lambda) =\frac{1_{|\lambda - \re z| > 1}}{|\lambda - z|},  \quad g(\lambda) = \frac{1}{|\lambda - z|^2}\]
so that
\[\ee \left| S(z)\frac{\partial}{\partial z} G(x,y;z) \right| \le \ee \iint \! \left(f(\lambda_1) + \tilde{f}(\lambda_1)\right) g(\lambda_2) \, \nu(d\lambda_1) |\mu_{xy}|(d\lambda_2).\]
Setting $I_\alpha = \re z +  [\alpha/N, (\alpha + 1)/N)$,
\begin{align*}&\ee \iint \! f(\lambda_1) g(\lambda_2) \, \nu(d\lambda_1) |\mu_{xy}|(d\lambda_2)\\
  & \le \sum_{\alpha, \beta \in \zz}  \left(\sup_{\lambda \in I_\alpha} f(\lambda) \right) \left(\sup_{\lambda \in I_\beta} g(\lambda) \right) \ee \left[ \nu(I_\alpha) |\mu_{xy}|(I_\beta) \right] \\
& \le \frac{C}{N^2}\sum_{\alpha, \beta \in \zz}  \left(\sup_{\lambda \in I_\alpha} f(\lambda) \right) \left(\sup_{\lambda \in I_\beta} g(\lambda) \right),
\end{align*}
where we used the second part of Lemma~\ref{thm:specavgext} to bound the expectations. Since $f$ and $g$ are symmetric about $\re z$ and monotone decreasing in $|\lambda - \re z|$, the previous chain of inequalities continues
\begin{align*}
&\le  \frac{4 C}{N^2} \sum_{\alpha, \beta \in \nn} f\left(\re z + \frac{\alpha}{N} \right) g\left(\re z + \frac{\beta}{N} \right) \\
&=  C N \sum_{\alpha = 0}^{N} \frac{1}{\sqrt{\alpha^2 + (N \, \im z)^2 }} \sum_{\beta \in \nn} \frac{1}{\beta^2 + (N \, \im z)^2 } \\
&\le C N \left( \log N + \frac{1}{N \, \im z} \right) \left(1 + \frac{1}{(N \, \im z)^{2}} \right).
\end{align*}
Finally, because $|\tilde{f}| \le 1$, the remaining summands satisfy
\[ \ee \iint \! \tilde{f}(\lambda_1) g(\lambda_2) \, \nu(d\lambda_1) |\mu_{xy}|(d\lambda_2)  \le \ee\int \!   \frac{1}{|\lambda - z|^2} \, |\mu_{xy}|(d\lambda) \le \frac{C}{\im z},\]
arguing as in~\eqref{eq:driftbound1}.
\end{proof}
Evaluating the trace defining $S(z)$ in the site basis,
\[S(z) = \frac{1}{N} \sum_y G(y,y;z),\]
we may average the bounds furnished by Theorem~\ref{thm:driftbounds} to obtain the following corollary, which gives the corresponding bounds for the drift of diffusion of the trace of the resolvent.
\begin{corollary} \label{thm:driftboundstr}There exists $C < \infty$, depending only on $C_V$, such that
\[\ee \left|\frac{1}{2N} \parder{^2}{z^2} S(z) \right| \le \frac{C}{N (\im z)^2} \]
and
\[\ee \left| S(z) \parder{}{z} S(z) \right| \le C N \left( \log N + \frac{1}{N \, \im z} \right) \left(1 + \frac{1}{(N \, \im z)^{2}} \right) +  \frac{C}{\im z}\]
for all $z \in \cp$.
\end{corollary}
We conclude this section with a bound in the same spirit as the previous results for a term which does not explicitly occur in the resolvent flow, but which will nevertheless prove useful in controlling the diffusion of \eqref{eq:approxflow}.
\begin{theorem}\label{thm:qvboundlemma}There exists $C < \infty$, depending only on $C_V$, such that
\[\ee \left[ \im G(x,x;z) \, \im S(z) \right]  \le C\left(1 + \frac{1}{N \, \im z} \right)^2\]
for all $x \in \{1,\dots,N\}$ and $z \in \cp$.
\end{theorem}
\begin{proof}
The proof follows along the same lines as that of Theorem~\ref{thm:driftbounds}. Setting $I_\alpha = \re z +  [\alpha/N, (\alpha + 1)/N)$, letting
\beq{eq:poissondef} P_z(\lambda) = \im \frac{1}{\lambda - z} = \frac{\im z}{(\lambda - \re z)^2 + (\im z)^2}
\eeq
denote the rescaled Poisson kernel, and using Lemma~\ref{thm:specavgext}, we see that
\begin{align*}\ee \left[ \im G(x,x;z) \, \im S(z) \right] &=  \ee \iint \! P_z(\lambda_1) P_z(\lambda_2) \, \nu(d\lambda_1) \mu_{x}(d\lambda_2)\\
&\le \sum_{\alpha, \beta \in \zz}  \left(\sup_{\lambda \in I_\alpha} P_z(\lambda) \right) \left(\sup_{\lambda \in I_\beta} P_z(\lambda) \right) \ee \left[ \nu(I_\alpha) \mu_{x}(I_\beta) \right]\\
&\le  \frac{C}{N^2}\sum_{\alpha, \beta \in \zz}  \left(\sup_{\lambda \in I_\alpha} P_z(\lambda) \right) \left(\sup_{\lambda \in I_\beta} P_z(\lambda) \right).
\end{align*}
Since $P_z$ is symmetric about $\re z$ and monotone decreasing in $|\lambda - \re z|$, the last term is in turn bounded by
\begin{align*}
&\le  \frac{4 C}{N^2} \sum_{\alpha, \beta \in \nn} P_z\left(\re z + \frac{\alpha}{N} \right) P_z\left(\re z + \frac{\beta}{N} \right) \\
&=  C  \sum_{\alpha \in \nn} \frac{N \, \im z}{\alpha^2 + (N \, \im z)^2 } \sum_{\beta \in \nn} \frac{N \, \im z}{\beta^2 + (N \, \im z)^2 } \\
&\le C\left(1 + \frac{1}{N \, \im z} \right)^2.
\end{align*}
\end{proof}

\section{Proof of Stability} \label{sec:flowsec}
In this section, we turn to the proofs of Theorems~\ref{thm:stabthmtr} and~\ref{thm:stabthm}. We start by deriving the stochastic differential equations~\eqref{eq:approxflow} for the resolvent $R_t(z)$ in terms of the Green functions and the normalized trace. For this, we define the martingales
\[dM_t(x,y;z) =  -\frac{1}{\sqrt{N}}\sum_{u \le v} \langle \delta_y, R_t(z) P_{uv} R_t(z) \delta_x \rangle \, dB_{uv}(t),\]
where
\[P_{uv} = \frac{1}{\sqrt{1 + \delta_{uv}}} \left( \ketbra{\delta_u}{\delta_v} + \ketbra{\delta_v}{\delta_u}\right) = \sqrt{N}\parder{}{B_{uv}}H_t\]
denotes the symmetric matrix element corresponding to $\{\delta_u,\delta_v\}$.
\begin{theorem} \label{thm:resflow} The Green function satisfies
\[dG_t(x,y;z) = \left(S_t(z) \parder{}{z} G_t(x,y; z) + \frac{1}{2N} \parder{^2}{z^2} G_t(x,y; z) \right) \, dt + dM_t(x,y;z)\]
for all $x,y \in \{1,\dots,N\}$ and $z \in \cp$.
\end{theorem}
\begin{proof} By the resolvent equation,
\[\parder{}{B_{uv}} R_t(z) = -\frac{1}{\sqrt{N}} R_t(z) P_{uv} R_t(z),\]
so using It\^{o}'s Lemma shows that
\begin{align*} dG_t(x,y;z)  &= \frac{1}{N} \sum_{u \le v} \langle \delta_y, R_t(z) P_{uv} R_t(z) P_{uv} R_t(z)\delta_x \rangle \, dt \\
&-  \frac{1}{\sqrt{N}}\sum_{u \le v} \langle \delta_y, R_t(z) P_{uv} R_t(z) \delta_x \rangle \, dB_{uv}(t) \\
&=  \frac{1}{N} \sum_{u \le v} \langle \delta_y, R_t(z) P_{uv} R_t(z) P_{uv} R_t(z)\delta_x \rangle \, dt  + dM_t(x,y;z).
\end{align*}
We expand the drift term as
\begin{align*}& \frac{1}{N} \sum_{u < v}  \langle \delta_y, R_t(z) \delta_v \rangle \langle \delta_u, R_t(z) \delta_u \rangle \langle \delta_v, R_t(z) \delta_x \rangle + \langle \delta_y, R_t(z) \delta_u \rangle \langle \delta_v, R_t(z) \delta_v \rangle \langle \delta_u, R_t(z) \delta_x \rangle \\
& + \frac{1}{N} \sum_{u < v}  \langle \delta_y, R_t(z) \delta_v \rangle \langle \delta_u, R_t(z) \delta_v \rangle \langle \delta_u, R_t(z) \delta_x \rangle + \langle \delta_y, R_t(z) \delta_u \rangle \langle \delta_v, R_t(z) \delta_u \rangle \langle \delta_v, R_t(z) \delta_x \rangle \\
&+ \frac{2}{N} \sum_u \langle \delta_y, R_t(z) \delta_u \rangle \langle \delta_u, R_t(z) \delta_u \rangle \langle \delta_u, R_t(z) \delta_x \rangle
\end{align*}
and exploit that the second term in each sum is the same as the first term with $u$ and $v$ interchanged to rewrite these sums as 
\begin{align*}&= \frac{1}{N}  \sum_{u, v}  \langle \delta_y, R_t(z) \delta_v \rangle \langle \delta_u, R_t(z) \delta_u \rangle \langle \delta_v, R_t(z) \delta_x \rangle\\
& + \frac{1}{N}  \sum_{u, v}  \langle \delta_y, R_t(z) \delta_v \rangle \langle \delta_u, R_t(z) \delta_v \rangle \langle \delta_u, R_t(z) \delta_x \rangle.
\end{align*}
In the second sum, we use that the spectral measures $\mu_{vu}$ are real to replace $ \langle \delta_u, R_t(z) \delta_v \rangle$ with $ \langle \delta_v, R_t(z) \delta_u \rangle$, which yields
\begin{align*}
&=\langle \delta_y, R_t(z)^2 \delta_x \rangle  \frac{1}{N} \tr R_t(z)  + \frac{1}{N} \langle \delta_y, R_t(z)^3 \delta_x \rangle\\
&= S_t(z) \parder{}{z} G_t(x,y; z) + \frac{1}{2N} \parder{^2}{z^2} G_t(x,y; z).
\end{align*}
\end{proof}
We remark that the applicability of the arguments in this paper to GUE perturbations in place of GOE perturbations is not affected by the last part of the proof, which made use of the fact that the spectral measures are real in the GOE case. This is because the additional unitary symmetry ensures that the third order term involving $\langle \delta_y, R_t(z)^3 \delta_x \rangle$ vanishes completely for the GUE flow.

By averaging the evolution of $G_t(x,x;z)$ over $x \in \{1,\dots,N\}$, we obtain an equation with a diffusion given by
\[M_t(z) = \frac{1}{N} \sum_x M_t(x,x;z),\]
which is the familiar complex Burgers equation for $S_t(z)$~\cite{MR2760897}.
\begin{corollary} \label{thm:resflowtr} The normalized trace satisfies
\[dS_t(z) = \left(S_t(z) \parder{}{z} S_t(z) + \frac{1}{2N} \parder{^2}{z^2} S_t(z) \right) \, dt + dM_t(z)\]
for all $z \in \cp$.
\end{corollary}

We will now employ the results of Section \ref{sec:smooth} to smooth the resolvent flow of Theorem \ref{thm:resflow}. Theorem~\ref{thm:driftbounds} and Corollary~\ref{thm:driftboundstr} already accomplish this for the drift, but some further analysis based on spatial averaging is required to control the diffusion and this is the content of the next two theorems. 
\begin{theorem}\label{thm:qvbounds} There exists a constant $C < \infty$, depending only on $C_V$, such that
\[\frac{1}{N} \sum_{y} \ee |M_t(x,y; z) | \le C\sqrt{ \frac{t }{N (\im z)^2} }  \left(1 + \frac{1}{N \, \im z} \right)\]
for all $x \in \{1,\dots,N\}$, $z \in \cp$ and $t \geq 0$.
\end{theorem}
\begin{proof} The quadratic variation of $M_t(x,y;z)$ satisfies 
\begin{align*} \langle M_t(x,y;z) \rangle &= \frac{1}{N} \int_0^t \! \sum_{u \le v} \left| \langle \delta_y, R_s(z) P_{uv} R_s(z) \delta_x \rangle \right|^2 \, ds \nonumber\\
&\le \frac{2}{N} \int_0^t \! \sum_{u, v} \left| \langle \delta_y, R_s(z) \delta_u \rangle \langle \delta_v, R_s(z) \delta_x \rangle \right|^2 \, ds \nonumber\\
&= \frac{2}{N} \int_0^t \! \left(\sum_{u} \left| \langle \delta_y, R_s(z) \delta_u \rangle \right|^2\right) \left(\sum_{v} \left| \langle \delta_v, R_s(z) \delta_x \rangle \right|^2\right) \, ds \nonumber\\
&= \frac{2}{N (\im z)^2} \int_0^t \! \im G_s(x,x;z) \, \im G_s(y,y;z) \, ds,
\end{align*}
where we combined the symmetrization argument of Theorem~\ref{thm:resflow} with the inequality $(a+b)^2 \le 2(a^2 + b^2)$. Hence
\begin{align*}
\frac{1}{N} \sum_y \ee |\langle  M_t(x,y;z) \rangle|  &\le \frac{2}{N (\im z)^2}  \int_0^t \! \ee \left[ \im G_s(x,x;z) \, \im S_s(z) \right]\, ds\\
&\le \frac{Ct }{N (\im z)^2}  \left(1 + \frac{1}{N \, \im z} \right)^2
\end{align*}
by Theorem~\ref{thm:qvboundlemma}. Combining the Burkholder-Davis-Gundy inequality with Jensen's inequality for $\frac{1}{N} \sum_y \ee$ shows that
\begin{align*}\frac{1}{N} \sum_{y} \ee |M_t(x,y; z) | &\le \! C\left( \frac{1}{N} \sum_y \ee  \langle M_t(x,y;z) \rangle \right)^{1/2}\\
& \le  C\sqrt{ \frac{t }{N (\im z)^2} }  \left(1 + \frac{1}{N \, \im z} \right).
\end{align*}
\end{proof}

Next, we state the corresponding result for the averaged martingale
\[M_t(z) = \frac{1}{N} \sum_x M_t(x,x;z)\]
occuring in Corollary~\ref{thm:resflowtr}.
\begin{theorem} \label{thm:qvboundstr} There exists a constant $C < \infty$, depending only on $C_V$, such that
\[\ee |M_t(z) | \le \sqrt{\frac{Ct}{N^2 (\im z)^3}}\]
for all $z \in \cp$ and $t \geq 0$.
\end{theorem}
\begin{proof} By symmetrization,
\begin{align*}M_t(z) &= \frac{1}{N} \sum_x M_t(x,x;z)\\
&= -\frac{1}{N^{3/2}} \sum_{u, v} \frac{1}{\sqrt{1 + \delta_{uv}}}  \int_0^t \! \sum_x \langle \delta_v, R_s(z) \delta_x \rangle \langle \delta_x, R_s(z) \delta_u\rangle \, dB_{uv}(s)\\
&=  -\frac{1}{N^{3/2}} \sum_{u, v} \frac{1}{\sqrt{1 + \delta_{uv}}} \int_0^t \! \parder{}{z} \langle \delta_v, R_s(z) \delta_u\rangle \, dB_{uv}(s),
\end{align*}
so the quadratic variation may be expressed as
\begin{align*}\langle M_t(z) \rangle &=  \frac{1}{N^3} \int_0^t \!  \sum_{u, v} \frac{1}{1 + \delta_{uv}} \left|\parder{}{z} \langle \delta_v, R_s(z) \delta_u\rangle \right|^2 \, ds\\
&\le \frac{1}{N^3(\im z)^2} \int_0^t \!  \sum_{u, v} \left|\langle \delta_v, R_s(z) \delta_u\rangle \right|^2 \, ds\\
&= \frac{1}{N^2 (\im z)^3} \int_0^t \!  \im S_s(z) \, ds.\\
\end{align*}
Using, in order, the Burkholder-Davis-Gundy inequality, Jensen's inequality, and the Wegner estimate~\eqref{eq:wegnerest} yields
\begin{align*}\ee |M_t(z) | &\le C \left(\ee  \langle M_t(z) \rangle \right)^{1/2}\\
& \le C\left( \frac{1}{N^2 (\im z)^3} \int_0^t \! \ee \im S_s(z) \, ds \right)^{1/2}\\
&\le \sqrt{\frac{Ct}{N^2 (\im z)^3}}.
\end{align*}
\end{proof}

The proofs of Theorems~\ref{thm:stabthmtr} and~\ref{thm:stabthm} now reduce to plugging the various previous estimates into the integrated forms of Theorem~\ref{thm:resflow} and Corollary~\ref{thm:resflowtr}. For the sake of completeness, we illustrate this with the proof of Theorem~\ref{thm:stabthm}, but omit the very similar proof of Theorem~\ref{thm:stabthmtr}.
\begin{proof}[Proof of Theorem~\ref{thm:stabthm}] By Theorem~\ref{thm:resflow},
\begin{align*} &\frac{1}{N}\sum_y \ee \left| G_t\left(x,y; E + i\eta\right)-  G_0\left(x,y;E + i\eta \right) \right|\\
 &\le  \frac{1}{N}\sum_y \int_0^t \! \ee \left| S_s(z) \parder{}{z} G_s(x,y; z) + \frac{1}{2N} \parder{^2}{z^2} G_s(x,y; z) \right| \, ds\\
 &+ \frac{1}{N}\sum_y \ee |M_t(x,y;z)|,
\end{align*}
which by Theorems~\ref{thm:driftbounds} and~\ref{thm:qvbounds} is bounded by
\begin{align*} &\le C t N \left( \log N + \frac{1}{N  \eta} \right) \left(1 + \frac{1}{(N  \eta)^{2}} \right) +  \frac{Ct}{\eta} + \frac{Ct}{N \eta^2}\\
 &+  C\sqrt{ \frac{t }{N \eta^2} }  \left(1 + \frac{1}{N  \eta} \right).
\end{align*}
After taking a factor $N^{-c/2}$ from $t \le N^{-(1 + c)}$ to control the $\log N$ term, each term is dominated by either $1 + (N\eta)^{-1}$ or $1 + (N\eta)^{-3}$, which proves the theorem.
\end{proof}

\section{Proof of Poisson Statistics} \label{sec:poisson}
In the remainder of this paper, we will show how to apply Theorems \ref{thm:stabthmtr} and \ref{thm:stabthm} to the ultrametric ensemble $H_n$ defined in \eqref{eq:Hdef}, thereby obtaining Theorems \ref{thm:poisson} and \ref{thm:localization}. When $ c > 0 $, the limit $ \lim_{n\to \infty} Z_{n,c}  \in (0,\infty)  $ exists, and thus we may drop the normalizing constant $ Z_{n,c} $ from the definition of  $ H_n $ without any loss of generality. Similarly to our approach in~\cite{MR3649447}, we will prove Theorem \ref{thm:poisson} by approximating $H_n  \equiv  \sum_{r=0}^n 2^{-\frac{1 + c}{2} r}\Phi_{n,r} $ with the truncated Hamiltonian
\beq{hnmdefinition} H_{n,m} =  \sum_{r=0}^m 2^{-\frac{1 + c}{2} r}\Phi_{n,r},
\eeq
which has the property that, for any $m \le k \le n$,
\beq{hnmdecomposition} H_{n,m} = \bigoplus_{j=1}^{2^{n-k}} H_{k,m}^{(j)},
\eeq
where each $H_{k,m}^{(j)}$ is an independent copy of $H_{k}$. Therefore
\[\mu_{n,m}(f) = \sum_{\lambda \in \sigma(H_{n,m})} f(2^n(\lambda - E)) \]
consists of $2^{n-m}$ independent components, a fact whose relevance to Theorem \ref{thm:poisson} is contained in the following characterization of Poisson point processes.
\begin{proposition} \label{poissoncharacterization}
Let $\{\mu_{n,j} \, | \, j = 1, \dots, N_n\}$ be a collection of point processes such that:
\begin{enumerate}
\item The point processes $\{\mu_{n,1}, \dots, \mu_{n, N_n}\}$ are independent for all $n \geq 1$.
\item If $B \subset \rr$ is a bounded Borel set, then
\[\lim_{n \to \infty} \sup_{j \le N_n} \pp(\mu_{n,j}(B) \geq 1) = 0.\]
\item There exists some $c \geq 0$ such that if $B \subset \rr$ is a bounded Borel set with $|\partial B| = 0$, then
\[\lim_{n \to \infty} \sum_{j=1}^{N_n} \pp(\mu_{n,j}(B) \geq 1) = c|B|\]
and
\[\lim_{n \to \infty} \sum_{j=1}^{N_n} \pp(\mu_{n,j}(B) \geq 2) = 0 .\]
\end{enumerate}
Then, $\mu_n = \sum_j \mu_{n,j}$ converges in distribution to a Poisson point process with intensity $c$.
\end{proposition}
We recall \cite{MR3364516} that a sequence of point processes $\mu_n$ converges in distribution to $\mu$ whenever
\[\lim_{n \to \infty} \ee e^{-\mu_n(P_z)}= \ee e^{-\mu(P_z)}\]
for all $z \in \cp$, where $P_z$ is the rescaled Poisson kernel
\beq{poissonkerneldef} P_z(\lambda) = \im \frac{1}{\lambda - z} = \frac{\im z}{(\lambda - \re z)^2 + (\im z)^2}.
\eeq
Hence, Theorem \ref{thm:poisson} follows by furnishing a sequence $m_n$ such that Proposition~\ref{poissoncharacterization} applies to $\mu_{n, m_n}$ and
\beq{needtoprove} \lim_{n \to \infty} \ee e^{-\mu_{n,m_n}(P_z)} = \lim_{n \to \infty} \ee e^{-\mu_n(P_z)}
\eeq
for all $z \in \cp$.

The difference $ H_n - H_{n,n-1} =  \sqrt{t\,}  \Phi_{n,n} $ is a Gaussian perturbation with time parameter $t = 2^{-(1+c) n }$. Therefore, Theorem \ref{thm:stabthmtr} shows that there exists $ C_z < \infty $ such that for all $\ell \geq n$ we have
\begin{align} \label{resflowbound1} \frac{1}{2^n}\, \ee \left| \tr (H_n - \zel)^{-1} - \tr (H_{n,n-1} - \zel)^{-1}\right| &\le C_z \, 2^{-\frac{c}{2}n-1}\,\left( 1+ 2^{3(\ell-n)}\right) \nonumber \\
&\leq C_z \, 2^{-\frac{c}{2}n} \,  2^{3(\ell-n) } 
\end{align}
with $\zel = E + 2^{-\ell}z$. Our strategy in achieving \eqref{needtoprove} thus consists of applying  \eqref{resflowbound1} to the finite-volume density of states measures
\[\nu_n(f) = 2^{-n} \tr f \left(H_n\right) ,   \quad \nu_{n,m}(f) = 2^{-n} \tr f \left(H_{n,m}\right) \]
in an iterative fashion.

\begin{theorem} \label{approximationbynunm} There exist $C_z < \infty$ and $\delta > 0$ such that
 \[ \ee \left| \nu_{n}\left(P_\zel\right) - \nu_{n,m}\left(P_\zel\right) \right| \le C_z\, 2^{3(\ell - (1+\delta) m)}\]
 for all $\ell \geq n$.
\end{theorem}
\begin{proof} The estimate \eqref{resflowbound1} proves that
\beq{thmappl} \ee \left| \nu_{k}\left(P_\zel\right) - \nu_{k,k-1}\left(P_\zel\right) \right| \le C_z\,2^{3(\ell - (1+\delta) k)}
\eeq
with $\delta = c/6$ when $\ell \geq k$. Since $\nu_n - \nu_{n,m}$ is given by a telescopic sum,
\[\nu_n(P_\zel) - \nu_{n,m}(P_\zel) =  \sum_{k=m+1}^{n} \left( \nu_{n,k}(P_\zel) -  \nu_{n,k-1}(P_\zel)\right),\]
the decomposition~\eqref{hnmdecomposition} implies that
\beq{telescopicterm} \nu_{n,k}(P_\zel) -  \nu_{n,k-1}(P_\zel)= 2^{-(n-k)} \sum_{j=1}^{2^{n-k}} \left( \nu_{k}(P_\zel) -  \nu_{k,k-1}(P_\zel) \right).
\eeq
Applying \eqref{thmappl} to each term in \eqref{telescopicterm} yields
\begin{align*} \ee \left| \nu_{n}\left(P_\zel\right) - \nu_{n,m}\left(P_\zel\right) \right| & \le \sum_{k=m+1}^{n} C_z\, 2^{3(\ell - (1 + \delta) k)} \le C_z\, 2^{3(\ell - (1 + \delta) m)}. 
\end{align*}
\end{proof}

Theorem \ref{approximationbynunm} has two important implications for the measures $\mu_n$ and $\mu_{n,m}$ which are based on the identities $\mu_n(P_z) =  \nu_n\left(P_{z_n}\right)$ and $\mu_{n,m}(P_z) =  \nu_{n,m}\left(P_{z_n}\right)$. The first of these enables us to find a suitable sequence $\mu_{n,m_n}$ satisfying \eqref{needtoprove}.
\begin{corollary}\label{divisibility}
There exists a sequence $m_n$ with $m_n \to \infty$ and $n - m_n \to \infty$ such that
\[\lim_{n \to \infty} \ee \left|\mu_n(P_z) -\mu_{n,m_n}(P_z) \right| = 0\]
for all $z \in \cp$.
\end{corollary}
\begin{proof} Since $\delta > 0$, there exists a sequence $m_n$ with $m_n \to \infty$, $n - m_n \to \infty$ and $n - (1 + \delta)m_n \to -\infty$. By applying Theorem~\ref{approximationbynunm} with $\ell = n$, we obtain
\[ \ee \left|\mu_n(P_z) - \mu_{n,m_n}(P_z)\right| \le C_z 2^{3(n - (1 + \delta)m_n)} \to 0.\]
\end{proof}

To finish the proof of Theorem \ref{thm:poisson}, we need to show that $\mu_{n,m_n}$ satisfies the hypothesis of Proposition~\ref{poissoncharacterization}. By~\eqref{hnmdecomposition}, $\mu_{n,m_n}$ is a sum of point processes
\[\mu_{n,m_n} = \sum_{j=1}^{2^{n-m_n}} \mu_{m_n,j}\]
with independent $\mu_{m_n, j}$. If $B \subset \rr$ is a bounded Borel set, the theorem of Combes-Germinet-Klein \cite{MR2505733} asserts that 
$\pp(\tr 1_B(H_{m}) \geq \ell) \le \left(C \, 2^m |B| \right)^\ell/ \ell! $
and hence for any $ \ell \geq 0 $:
\begin{equation}\label{cgkbound}
\pp(\mu_{m_n, j}(B) \geq \ell) \le \frac{(C |B| \, 2^{m_n-n})^\ell}{\ell!}.
\end{equation}
Since $n-m_n \to \infty$, the first hypothesis of Proposition~\ref{poissoncharacterization} follows. Writing
\[X(n, \ell) = \sum_{j=1}^{2^{n-m}} \pp(\mu_{m_n, j}(B) \geq \ell),\]
~\eqref{cgkbound} implies
\[X(n,\ell) \le 2^{n-m_n} \frac{(C |B|\, 2^{m_n-n})^\ell}{\ell!} \to 0\]
when $\ell \geq 2$. In particular, $X(n, 2) \to 0$ and the last hypothesis of Proposition~\ref{poissoncharacterization} is satisfied. It remains to prove the remaining hypothesis of Proposition~\ref{poissoncharacterization}, which is the second important consequence of Theorem \ref{approximationbynunm} and is contained in the following theorem.
\begin{theorem}  Let $B\subset \mathbb{R} $ be a bounded Borel set. Then,
\[\lim_{n \to \infty} X(n,1) = \nu(E)|B|.\]
\end{theorem}
\begin{proof} By \eqref{hnmdecomposition} we have $ \ee \nu_{p,n} = \ee \nu_{n} $ for any $ p \geq n$, and so we conclude from Theorem~\ref{approximationbynunm} with $\ell = n$ that
\begin{align*}\lim_{n \to \infty} \left|\ee \left[\nu_n(P_{z_n}) - \nu(P_{z_n})\right]\right| &= \lim_{n \to \infty}\lim_{p \to \infty} \left|\ee \left[ \nu_n(P_{z_n}) - \nu_p(P_{z_n}) \right]\right|\\
&= \lim_{n \to \infty}\lim_{p \to \infty} \left| \ee \left[ \nu_{p,n}(P_{z_n}) - \nu_p(P_{z_n}) \right]\right|\\
&\le  \lim_{n \to \infty} C_z \, 2^{-3\delta n} = 0.
\end{align*}
This shows that the measures $\lambda_n(B) = 2^n \nu(2^{-n}B + E)$ satisfy
\[\lim_{n \to \infty}\left( \ee \mu_n(P_z) - \lambda_n(P_z)\right)= 0\]
and Corollary \ref{divisibility} implies that also
\begin{equation}\label{intensityconvergence}
\lim_{n \to \infty}\left(\ee \mu_{m_n}(P_z) - \lambda_n(P_z)\right)= 0.
\end{equation}
For any bounded Borel set $B \subset \rr$, the indicator $1_B$ is in the $L^1$-closure of the finite linear combinations from the set $\{P_z \, | \, z \in \cp\}$ and the measures $\ee \mu_n$ are absolutely continuous with uniformly bounded densities by the Wegner estimate. Together, these two observations yield that \eqref{intensityconvergence} is valid for any bounded Borel set $B \subset \rr$. Moreover, since $E$ is a Lebesgue point of $\nu$,
\[\lim_{n \to \infty} \lambda_n(B) = \lim_{n \to \infty} 2^n \nu(2^{-n}B + E) = \nu(E)|B|,\]
and hence we have shown that
\beq{indicatorconvergence} \lim_{n \to \infty} \ee \mu_{m_n}(B) = \nu(E)|B|.
\eeq
Since $\mu_{n_m, j}(B)$ takes values in the non-negative integers
\[\lim_{n \to \infty} X(n,1) = \lim_{n \to \infty} \sum_{j=1}^{2^{n - m}} \ee \mu_{n_m, j}(B) - \lim_{n \to \infty} \sum_{\ell \geq 2} X(n,\ell)\]
so ~\eqref{cgkbound}, \eqref{indicatorconvergence} and the dominated convergence theorem give
\[\lim_{n \to \infty} X(n,1) = \lim_{n \to \infty} \sum_{j=1}^{2^{n - m}} \ee \mu_{n_m, j}(B) = \nu(E)|B|.\]
\end{proof}

\section{Proof of Eigenfunction Localization} \label{sec:localization}
In this section, we prove Theorem \ref{thm:localization} by comparing the eigenfunctions of $H_n$ with the obviously localized eigenfunctions of $H_{n,m}$. Nevertheless, we again start by considering a more general $N \times N$ random matrix $H = T+V$ with a potential satisfying~\eqref{eq:potassumption} and proving an implication of local resolvent bounds for the eigenfunction correlator
\[Q(x,y;W) = \sum_{\lambda \in \sigma(H) \cap W} |\psi_\lambda(x)\psi_\lambda(y)|\]
in some mesoscopic spectral window
\[W = \left[E_0 - N^{-(1-w)}, E_0 + N^{-(1-w)}\right]\]
with $w > 0$.

\begin{theorem} \label{thm:qbound} Let $\eta = N^{-(1+\ell)}$ with $\ell > w > 0$ and let $Y \subset \{1,\dots,N\}$. Then, there exists a constant $C < \infty$, depending only on $C_V$, such that
\[\pp\left(\sum_{y \in Y} Q(x,y;W) > \frac{2}{\pi}\sum_{y \in Y} \int_W \! \left| \im G(x,y;E+i\eta) \right| \, dE + \frac{\log N}{N^w}\right) \le CN^{w-\ell}\]
for all $x \in \{1,\dots,N\}$.
\end{theorem}

The proof of Theorem \ref{thm:qbound} is based on the following two lemmas, the first of which is formulated in terms of the the Poisson kernel $P_z$ defined in~\eqref{eq:poissondef}.
\begin{lemma}\label{thm:qlemma1} There exists a constant $ C < \infty $, depending only on $C_V$, such that
\[\ee \sum_{y} |\mu_{xy}|\left(1_{W^c} (1_W \ast P_{i\eta})\right)  \le C  N\eta \left( 1+ \log\sqrt{ 1 + \eta^{-2} |W|^2 } \right)\]
for all intervals $ W \subset \rr$ and $ \eta > 0 $.
\end{lemma}
\begin{proof} By spectral averaging (Lemma \ref{thm:specavgext}),
\begin{align*}\sum_{y}\ee  |\mu_{xy}|(1_{W^c} (1_W \ast P_{i\eta})) & \le CN\int_{W^c} (1_{W} \ast P_{i\eta})(\lambda) \, d\lambda\\
&= CN\int_{W^c}\!\int_W\! \frac{\eta}{(u-v)^2 + \eta^2} \, du\,dv\\
&=CN \eta \int_{\eta^{-1} W^c}\!\int_{\eta^{-1} W}\! \frac{1}{1 + (u-v)^2} \, du\,dv.
\end{align*}
Without loss of generality, we may assume that $\eta^{-1} W = [-a,a]$, so
\begin{align*}
 \int_{\eta^{-1} W^c}\!\int_{\eta^{-1} W}\! \frac{1}{1 + (u-v)^2} \, du\,dv &= \int_{\eta^{-1}W^c} \! \arctan(a-v)-  \arctan(v+a) \,dv\\
&= 2 \int_{a}^\infty \!  \arctan(v+a) -  \arctan(v-a  ) \,dv
\end{align*}
since $\arctan v$ is an odd function of $v$. After the appropriate translations, this last integral is 
\begin{align*}
&= 2 \lim_{R \to \infty}\int_{R-a}^{R+a} \arctan v \, dv-  2\int_0^{2a} \! \arctan v \, dv \\
&= 2 \left(\frac{2\pi a}{2} -\int_0^{2a} \! \arctan v \, dv\right)\\
&= 2a\left(\frac{\pi}{2} - \arctan(2a)\right) +  \log\sqrt{1 + 4a^2} .
\end{align*}
The proof is completed by noting $|\arctan(x) - \pi/2| \le 1/x$ and inserting $a=\eta^{-1}|W|/2$.
\end{proof}

The second lemma needed for the proof of Theorem \ref{thm:qbound} controls the generic spacing between the eigenvalues of $H$ in the interval $W$.

\begin{lemma}\label{thm:qlemma2} Let $W \subset \rr$ be an interval and $|W| \geq S > 0$. Then, there exists a constant $C < \infty$, depending only on $C_V$, such that the event
\[\mathcal{E} = \ \left\{ \min_{\lambda \in \sigma(H) \cap W} d\left(\lambda,  \partial W \cup \sigma(H) \setminus \{\lambda\}\right) > 2S \right\} \]
satisfies
\[\pp(\mathcal{E}^c) \le CSN (1 + |W|N).\]
\end{lemma}
\begin{proof} We split $W$ into a disjoint union of adjacent intervals
\[W = I_1 \cup \dots \cup I_p\]
with $|I_k| = 2S$ for $1 \le k \le p-1$ and $|I_{p}| \le 2S$, and let $\tilde{I}_k$ denote the fattened interval $I_k + [-2S, 2S]$. Then $\mathcal{E}^c$ can only occur if
\begin{enumerate}
\item $\tilde{I}_k$ contains at least two eigenvalues of $H$ for some $1 \le k \le p$, or
\item $\partial W + [-2S,2S]$ contains an eigenvalue of $H$.
\end{enumerate}
Therefore, the Wegner and Minami estimates show that
\begin{align*}\pp({\mathcal{E}}^c) &\le \pp\left(\left|\left( \partial W + [-2S,2S] \right)\cap \sigma(H)\right| \geq 1\right) +  \sum_{k=1}^p \pp\left( \left|\tilde{I}_k \cap \sigma(H)\right| \geq 2\right)\\
&\le CSN+ Cp\left(S  N \right)^2,
\end{align*}
and since $p \le 2|W|/S$, this proves the lemma.
\end{proof}

\begin{proof} [Proof of Theorem \ref{thm:qbound}]
Let $S = \frac{8}{\pi}\eta$ so that the event $\mathcal{E}$ defined in Lemma \ref{thm:qlemma2} satisfies
\[\pp(\mathcal{E}^c) \le C N^{w-\ell}.\]
Since the spectral measures $\mu_{xy}$ are real, we can construct the function
\[f(E) = \sum_{\lambda \in \sigma(H) \cap W} \sgn \left[ \psi_\lambda(x)\psi_\lambda(y) \right] I_{\lambda}(E),\]
where $I_\lambda$ denotes the indicator function of the interval $[\lambda-S, \lambda+S]$. We will prove that on the event $\mathcal{E}$ we have $\|f\|_\infty \le 1$ and
\beq{eq:variationpoisson} \sum_{y \in Y} |\mu_{xy}|(W) \le \frac{2}{\pi}\sum_{y \in Y} \mu_{xy}(f \ast P_{i\eta}) +  \sum_{y} |\mu_{xy}|\left(1_{W^c} (1_W \ast P_{i\eta})\right),
\eeq
so, since
\begin{align*} \mu_{xy}(f \ast P_{i\eta}) &= \iint \! f(E) P_{\lambda + i\eta}(E) \, dE \, \mu_{xy}(d\lambda) \\
&=  \int \! f(E) \int \! P_{E + i\eta}(\lambda) \, \mu_{xy}(d\lambda) \, dE \\
&\le \|f\|_\infty \int_W \! \left| \im G(x,y; E + i\eta) \right| \, dE,
\end{align*}
the theorem follows from Lemma~\ref{thm:qlemma1} and Markov's inequality.

On $\mathcal{E}$, the intervals $I_\lambda$ are disjoint and contained in $W$, so $|f| \le 1_W$ and, in particular, $\|f\|_\infty \le 1$. To verify \eqref{eq:variationpoisson}, we note that
\begin{align*}
 \mu_{xy}(f \ast P_{i\eta}) &= \mu_{xy}\left(1_{W} (f \ast P_{i\eta})\right) + \mu_{xy}\left(1_{W^c} (f \ast P_{i\eta})\right)\\
&\geq   \sum_{\lambda \in \sigma(H) \cap W} \psi_\lambda(x) \psi_\lambda(y) (f \ast P_{i\eta})(\lambda)  -  |\mu_{xy}|\left(1_{W^c} (1_W\ast P_{i\eta})\right)
\end{align*}
on $\mathcal{E}$ and hence it remains only to prove that
\[\sgn \left[ \psi_\lambda(x)\psi_\lambda(y) \right] (f \ast P_{i\eta})(\lambda) \geq \frac{\pi}{2}\]
for all $\lambda \in \sigma(H) \cap W$. This is based on the fact that
\[\int \! (1-I_\lambda(E)) P_{\lambda + i\eta}(E) \, dE \le \frac{2\eta}{S} = \frac{\pi}{4}\]
and hence
\[\int \! I_\lambda(E) P_{\lambda + i\eta}(E) \, dE \geq \pi - \int \! (1- I_\lambda(E)) P_{\lambda + i\eta}(E) \, dE \geq \frac{3\pi}{4}. \]
If $\lambda \in \sigma(H) \cap W$ with $\sgn \left[ \psi_\lambda(x)\psi_\lambda(y) \right] = 1$, it follows that
\begin{align*} (f \ast P_{i\eta})(\lambda)  &= \int \! f(E) P_{\lambda+i\eta}(E) \, dE\\
&\geq \int \! I_\lambda(E) P_{\lambda+i\eta}(E) \, dE  -  \int \! (1-I_\lambda(E))P_{\lambda+i\eta} (E) \, dE \\
&\geq \frac{\pi}{2},
\end{align*}
and similarly
\[(f \ast P_{i\eta})(\lambda) \le - \frac{\pi}{2}\]
if $\sgn \left[ \psi_\lambda(x)\psi_\lambda(y) \right] = -1$.
\end{proof}

The proof of the last theorem made use of the fact that the spectral measures $\mu_{xy}$ are always real for the GOE flow. It is possible to extend this result to models with complex off-diagonal spectral measures, such as the GUE flow, by using the fact that
\[\langle \delta_y, \im (H-z)^{-1}\delta_x \rangle +  \langle \delta_x, \im (H-z)^{-1}\delta_y \rangle =  \im G(x,y;z) + \im G(y,x;z),\]
but we omit these complications here.

With Theorem \ref{thm:qbound} in hand, we now turn to the proof of Theorem \ref{thm:localization}. As in Section~\ref{sec:poisson}, we drop the normalizing constant $ Z_{n,c} $ from the definition of $ H_n$. The core of this argument again consists of resolvent bounds for Gaussian perturbations, and thus we consider the Green functions
\[G_n(x,y;z) = \langle \delta_y, (H_n - z)^{-1} \delta_x \rangle , \quad G_{n,m}(x,y;z) =  \langle \delta_y, (H_{n,m} - z)^{-1} \delta_x \rangle.\]
If $\eta = 2^{-(1 + \ell)n}$ for some $\ell > 0$, Theorem \ref{thm:stabthm} proves that there exists $ C < \infty $ such that
\begin{align*} & 2^{-k}\sum_{y \in B_k(x)} \ee \left| G_k\left(x,y; E+i\eta\right)-  G_{k,k-1}\left(x,y;E+i\eta\right) \right|\\
& \le C \, 2^{-\frac{c}{2}k} \left(1 + 2^{3( (1+\ell) n - k)}\right)\\
&= C\,2^{3(1 + \ell)n-3(1+\delta)k}
\end{align*}
with $\delta = c/6$ whenever $k \le n$. Iterating this result, we see that
\begin{align*} & 2^{-n} \sum_{y \in B_n} \ee \left| G_n\left(x,y; E+i\eta\right)-  G_{n,m}\left(x,y;E+i\eta \right) \right|\nonumber\\
&\le 2^{-n} \sum_{k=m+1}^n\sum_{y \in B_n} \ee \left| G_{n,k}\left(x,y; E+i\eta\right)-  G_{n,k-1}\left(x,y;E+i\eta \right) \right|\nonumber\\
&= 2^{-n} \sum_{k=m+1}^n \sum_{y \in B_k(x)} \ee \left| G_{k}\left(x,y; E+i\eta\right)-  G_{k,k-1}\left(x,y;E+i\eta \right) \right|\nonumber\\
&\le 2^{-n} \sum_{k=m+1}^n 2^k C\,2^{3(1 + \ell)n-3(1+\delta)k} \le C\,2^{(3(1+\ell) - 1) n} \, 2^{-(3(1+\delta)-1)m}.
\end{align*}
Since $\delta > 0$, we can choose $\ell > 0$, $\epsilon \in (0,1)$, and $w \in (0, \ell)$ such that
\[2\mu \vcentcolon= (1 - \epsilon)(3(1 + \delta)-1) - (3(1+\ell) -1)  - w > 0.\]
Thus, setting $m_n = (1 - \epsilon)n$ and
\[W = \left[E - 2^{-(1-w)n}, E + 2^{-(1-w)n}\right],\]
and using that $G_{n,m}(x,y;z) = 0$ if $y \notin B_m(x)$ show that
\[\sum_{y \in B_n \setminus B_{m_n}(x)} \ee \int_W \!  \left| \im G_n\left(x,y; E+i\eta\right) \right| \, dE \le C\,2^{-2\mu n}.\]
Applying Markov's inequality, we arrive at
\[\pp\left(\sum_{y \in B_n \setminus B_{m_n}(x)} \int_W \!  \left| \im G_n\left(x,y; E+i\eta\right) \right| \, dE > 2^{-\mu n}\right) \le C\, 2^{-\mu n},\]
so Theorem \ref{thm:localization} follows from Theorem \ref{thm:qbound}, which says that
\[ \sum_{y \in B_n \setminus B_{m_n}(x)} Q_n(x,y;W) \le \sum_{y \in B_n \setminus B_{m_n}(x)} \int_W \! \left| \im G_n(x,y;E+i\eta) \right| \, dE + \frac{\log 2^n}{2^{wn}}\]
with probability $1 - \oh\left(2^{(w-\ell)n}\right)$.

\appendix
\section{Proofs for the Rosenzweig-Porter model}\label{sec:RosPorter}
\begin{proof}[Proof of Theorem~\ref{thm:rpthm}]  As $N \to \infty$, the random measure defined by
\[\mu_{N,0}(f) = \sum_{\lambda \in \sigma(H_0)} f(N(\lambda - E))\]
converges in distribution to a Poisson point process with intensity $\varrho(E)$. Setting $z_N = E + z/N$, a simple calculation yields
\[\mu_N(P_z) = \im S_t(z_N).\]
Thus,
\[\left|\ee e^{-\mu_N(P_z)} - \ee e^{-\mu_{N,0}(P_z)}\right| \le \ee |S_t(z_N) - S_0(z_N)| \le C N^{-c/2},\]
which shows that the characteristic functionals of $\mu_N$ and $\mu_{N,0}$ asymptotically agree on the set $\{P_z: z \in \cp\}$ whose linear span is dense in $C_0$. This proves the first point.

For the second assertion, choose $\ell > w >0$ and $\mu_0 > 0$ such that
\[3\ell + w + 2\mu_0 \le c/2.\]
Since $G_0(x,y; z) = 0$ for $x \neq y$, Theorem~\ref{thm:stabthm} shows that with $\eta = N^{-(1+\ell)}$ we have
\begin{align*} \ee \sum_{y \neq x} \int_W \! \left| \im G_t(x,y; E + i\eta) \right|  \, dE  &\le C |W| N  N^{-c/2} (\eta N)^{-3}\\
&\le C N^{w + 3\ell - c/2} \le N^{-2\mu_0}.
\end{align*}
By Markov's inequality,
\[ \pp \left( \sum_{y \neq x} \int_W \! \left| \im G_t(x,y; E + i\eta) \right|  \, dE  \geq  N^{-\mu_0} \right) \le CN^{-\mu_0} \]
so choosing $0 < \mu < \min\{w, \mu_0\}$ and $\kappa = \min\{w-\ell, \mu_0\}$, Theorem~\ref{thm:qbound} shows that
\[ \pp\left(\sum_{y \neq x} Q_N(x,y; W) > N^{-\mu} \right) \le CN^{-\kappa}. \]
\end{proof}

\section*{Acknowledgements} We thank Y. V. Fyodorov for introducing us to the ultrametric ensemble. This work was supported by the DFG (WA 1699/2-1).

\bibliographystyle{abbrv}
\bibliography{References}
\bigskip
\begin{minipage}{0.5\linewidth}
\noindent Per von Soosten\\
Zentrum Mathematik, TU M\"{u}nchen\\
Boltzmannstra{\ss}e 3, 85747 Garching\\
Germany\\
\verb+vonsoost@ma.tum.de+
\end{minipage}%
\begin{minipage}{0.5\linewidth}
\noindent Simone Warzel\\
Zentrum Mathematik, TU M\"{u}nchen\\
Boltzmannstra{\ss}e 3, 85747 Garching\\
Germany\\
\verb+warzel@ma.tum.de+
\end{minipage}
\end{document}